%% file: main.tex
\documentclass[12pt,twoside]{amsart}
\input{includeNice3}
\usepackage{mabliautoref}
\usepackage{amssymb}
\usepackage{amscd}
\usepackage[abbrev,alphabetic]{amsrefs}
\usepackage{csquotes}
\usepackage[a4paper,margin=1in]{geometry}  
\usepackage{CJKutf8}
\usepackage{soul}

\usepackage{mathrsfs}
\usepackage{caption}

\usepackage{url}
\RequirePackage{booktabs, multirow}
\RequirePackage{pgf}

\theoremstyle{plain}

\title[Deformations of fibered Calabi--Yau varieties]
{Deformations of fibered Calabi--Yau varieties}

 \author[B. Bakker]{Benjamin Bakker} 
\address{Department of Mathematics, Statistics, and Computer Science, University of Illinois at Chicago, 851 S. Morgan Street, Chicago, IL 60607, USA}
\email{bakker.uic@gmail.com}

 \author[K. DeVleming]{Kristin DeVleming} 
\address{Department of Mathematics, University of California San Diego, 2985 Muir Lane, La Jolla, CA 92093, USA}
\email{kedevleming@ucsd.edu}

\author[S. Filipazzi]{Stefano Filipazzi}
\address{Department of Mathematics, Duke University, 120 Science Drive, 117 Physics Building, Campus Box 90320, Durham, NC 27708-0320, USA}
\email{stefano.filipazzi@duke.edu}

\author[R. Laza]{Radu Laza} 
\address{Department of Mathematics, Stony Brook University, Stony Brook, NY 11794, USA}
\email{radu.laza@stonybrook.edu}

\author[J. Li]{Jennifer Li}
\address{Department of Mathematics, Princeton University, Fine Hall, 304 Washington Road, Princeton, NJ 08540, USA}
\email{jenniferli@princeton.edu}

\author[R. Svaldi]{Roberto Svaldi}
\address{Dipartimento di Matematica ``F. Enriques'', Universit\`a degli Studi di Milano, Via Saldini 50, 20133 Milano (MI), Italy}
\email{roberto.svaldi@unimi.it}

\author[C. Wang]{Chengxi Wang}
\address{Yau Mathematical Sciences Center,
Jingzhai, Tsinghua University, Haidian District,
Beijing, China, 100084}
\email{chxwang@tsinghua.edu.cn}

\author[J. Zhao]{Junyan Zhao}
\address{Department of Mathematics, University of Maryland, William E. Kirwan Hall, 4176 Campus Dr, College Park, MD 20742, USA}
\email{jzhao81@umd.edu}



\def\O#1.{\mathcal {O}_{#1}}			
\def\pr #1.{\mathbb P^{#1}}				
\def\af #1.{\mathbb A^{#1}}			
\def\ses#1.#2.#3.{0\to #1\to #2\to #3 \to 0}	
\def\xrar#1.{\xrightarrow{#1}}			
\def\K#1.{K_{#1}}						
\def\bA#1.{\mathbf{A}_{#1}}			
\def\bM#1.{\mathbf{M}_{#1}}
\def\bN#1.{\mathbf{N}_{#1}}
\def\bL#1.{\mathbf{L}_{#1}}				
\def\bB#1.{\mathbf{B}_{#1}}				
\def\bK#1.{\mathbf{K}_{#1}}			
\def\subs#1.{_{#1}}					
\def\sups#1.{^{#1}}						
\def \neone#1.{\overline{{\rm NE}(#1)}}

\def\Def{\operatorname{Def}}
\def\gr{\operatorname{gr}}

\begin{document}

\begin{abstract}
    Koll\'{a}r showed in \cite{kollar15} that 
    small deformations of elliptically fibered smooth $K$-torsion varieties with $H^2(X,\mathcal{O}_X)=0$ remain elliptically fibered. We extend this result to any fibered smooth $K$-torsion variety $X$ with $H^2(X,\mathcal{O}_X)=0$, using Hodge theoretic techniques and the $T^1$-lifting criterion of Kawamata--Ran \cites{ran, Kaw92}. More generally, our strategy implies that even without the cohomological assumption, small deformations of a semiample line bundle on a smooth $K$-torsion variety remain semiample up to homological equivalence.
\end{abstract}

\maketitle

\tableofcontents

\input{introduction_arxiv_v1}

\subsection*{Acknowledgments}
This work began at the workshop ``Higher dimensional log Calabi--Yau varieties'' held in 2024 at the American Institute of Mathematics organized by Yoshinori Gongyo, Mirko Mauri, Joaqu\'in Moraga, and the sixth named author.  We are indebted to the institute for their support and hospitality.  We are very grateful to Paul Hacking for bringing our attention to this problem, and for a number of enlightening conversations.  We warmly thank Patrick Brosnan, Yoshinori Gongyo, Mirko Mauri, and Ziquan Zhuang for several insightful conversations.

BB was partially supported by NSF grant DMS-2401383. 
KD was partially supported by NSF grant DMS-2302163.
SF was partially supported by the ERC starting grant $\#$804334 and by Duke University. 
RL was partially supported by NSF grant DMS-2502134.
JL was supported by the Simons Foundation grant SFI-MPS-MOV-00006719-02.
RS was partially supported by the “Programma per giovani ricercatori Rita Levi Montalcini” of MUR and by PSR 2022 – Linea 4 of the University of Milan.  He is a member of the GNSAGA group of INDAM.
JZ was supported by a Simons travel grant.
    
\input{deformation_arxiv_v1}

\input{subvarieties_arxiv_v1}

\bibliographystyle{amsalpha}
\bibliography{refs}
	
\end{document}

%% file: introduction_arxiv_v1.tex
\section{Introduction}

In this article, we work over the field of complex numbers $\mathbb{C}$.
By the Lefschetz principle, all results extend to germs of algebraic families defined over any algebraically closed field of characteristic $0$, provided the assumptions in the corresponding statements are satisfied.

\medskip

The class of \emph{$K$-torsion varieties},
i.e., varieties with numerically trivial canonical line bundle, plays a central role in the birational classification of algebraic varieties. According to the classification framework from the Minimal Model Program, up to birational equivalence, projective varieties are expected to admit a decomposition via fibrations---possibly to a point---whose general fibers have canonical class that is either ample, numerically trivial, or anti-ample. These three classes of varieties demonstrate markedly different behavior. \emph{Canonically polarized varieties} and \emph{Fano varieties}—those for which the canonical class is ample and anti-ample, respectively—are, by definition, endowed with a natural ample line bundle. This line bundle yields a projective embedding and often serves as a first step toward understanding their classification and behavior in families.
The class of $K$-torsion varieties is approached with different methods.

Within the class of $K$-torsion varieties, the Beauville--Bogomolov decomposition theorem (cf.\ \cites{Bog74, Bea83}) asserts that any projective manifold with numerically trivial canonical class admits a finite \'etale cover by products of abelian varieties, irreducible holomorphic symplectic manifolds, and strict Calabi--Yau manifolds. With this result, it is natural to study the birational geometry, deformation theory, and moduli of $K$-torsion varieties by studying those for varieties belonging to each of the three classes appearing in the Beauville--Bogomolov decomposition.

\medskip

A natural problem in the study of $K$-torsion varieties concerns the existence and structure of \textit{fibrations} on these varieties, i.e., a surjective morphism 
$f \colon X \to Y$ 
with connected fibers, from a $K$-torsion variety 
$X$ onto a projective variety $Y$. Such morphisms encode many important structural properties of $X$, and they naturally appear in other studies of $K$-torsion varieties, particularly, when $\dim Y < \dim X$.
For example, 
to every elliptic Calabi-Yau threefold $X$ one can
associate a six-dimensional 
$N=1$ 
supergravity theory obtained as the low-energy limit of F-theory compactified on 
$X$,
(cf. 
\cites{Vafa, MV1, MV2, Donagi}), 
and Lagrangian fibrations on irreducible symplectic manifolds are algebraic realizations of the Strominger--Yau--Zaslow fibration, one of the structures at the heart of mirror symmetry for symplectic varieties (cf. \cites{Mat99, Gross_SYZ}).

\medskip

In projective algebraic geometry,
fibrations over lower dimensional varieties are naturally induced by semiample divisors that are not big. Indeed, given any semiample line bundle $L$ on a variety $X$, the global sections of $L^{\otimes m}$ for some $m\geq1$ yield a morphism $X \to \mathbb{P}H^0(X, L^{\otimes m})$;
when $L$ is not big, the image of $X$ via this morphism has strictly lower dimension than $X$.  
Conversely, given a fibration $f: X \to Y$ where $\dim Y < \dim X$, the pullback of any ample line bundle on $Y$ will be a semiample divisor on $X$ that is not big.  In the special case of $K$-torsion varieties, the generalized abundance conjecture predicts that every nef line bundle is semiample up to numerical equivalence (see, e.g., \cite{Oguiso93}, \cite{MP97}*{Chapter 10}, \cite{kollar15}*{Conjecture 51}, or  \cite{LP20}).
Thus, the numerical class of any nef line bundle on a $K$-torsion variety is expected to induce a fibration.  

\medskip

Given a fibration $f\colon X \to Y$ on a $K$-torsion variety $X$, it is then natural to ask if the fibration structure deforms in families, i.e., given a deformation $\mathcal{X} \to T$ of $\mathcal{X}_0 = X$, does every fiber $\mathcal{X}_t$ admit a fibration?  If $L$ is the semiample line bundle on $X$ inducing the fibration $f$, does $L$ deform to a relatively semiample line bundle on $\mathcal{X}$ inducing a fibration on each fiber?  The answer to this question is negative in general.
For instance, the product of any abelian variety and an elliptic curve admits a fibration by elliptic curves via the first projection;
however, a general algebraic deformation will be a simple
abelian variety which admits no fibration by lower dimensional abelian varieties. Similarly, certain K3 surfaces admit elliptic fibrations, but a very general algebraic K3 surface has Picard rank 1, 
and thus a very general algebraic deformation of an elliptically fibered K3 surface admits no elliptic fibration. 
Nonetheless, under suitable cohomology vanishing assumptions in the $K$-torsion case, elliptic fibrations indeed deform.  

\medskip

Building on earlier work of Wilson \cites{wilsoni,wilsonii} in the three-dimensional case, Koll\'ar \cite{kollar15} proved that all small deformations of an elliptically fibered $K$-torsion variety $f\colon X \to Y$ remain elliptically fibered, assuming the cohomological vanishing $H^2(X,\mathcal O_X)=0$.  The first main result of this paper is the following generalization of Koll\'ar's result, which provides a general affirmative answer to the deformation problem for arbitrary fibrations for any $K$-torsion variety $X$ 
satisfying the cohomological vanishing $H^2(X, \mathcal O_X)=0$.

\begin{theorem}[cf.\ Corollary~2.4]
\label{main.thm.intro_CY}
Let 
$\pi \colon \mathcal{X} \to T$ 
be a smooth projective family of $K$-torsion varieties over an analytic germ $(0\in T)$. Assume that the fiber $\mathcal X_0 \coloneqq \pi^{-1}(0)$
satisfies
$H^2(\mathcal{X}_0,\mathcal O_{\mathcal X_0})=0$,
and that there exists a fibration
$\psi_0 \colon \mathcal X_0 \to \mathcal Y_0$ with $\cY_0$ projective.
Then there exists a commutative diagram of projective morphisms
\begin{align*}
    \xymatrix{
\mathcal X \ar[rr]^{\psi} \ar[dr]_{\pi}
& 
& 
\mathcal Y \ar[dl]^{\phi}
\\
& 
T
& 
    }
\end{align*}
such that 
$\phi$ is a flat deformation of 
$\mathcal Y_0$ and $\psi\vert_{\mathcal X_0}=\psi_0$.
\end{theorem}

In the setup of \autoref{main.thm.intro_CY}, the vanishing condition 
$H^2(\cX_0,\mathcal O_{\cX_0})=0$
persists along any algebraic deformation of $\cX_0$ due to the upper semicontinuity of cohomology. This condition is used to ensure that any line bundle can be lifted to the total space $\mathcal X$.  Other than the elliptic case, \autoref{main.thm.intro_CY} was known for K3-fibrations of strict Calabi--Yau threefolds, but already unknown for abelian surface fibrations of threefolds (see \cite{wislonmin}*{\S 1}).

\medskip

The proof of \autoref{main.thm.intro_CY} is reduced to proving the claimed properties hold for the Kuranishi family over the miniversal deformation space
$\Def(\mathcal X_0)$ of 
$\mathcal X_0$:
indeed, any smooth proper family 
$\mathcal X \to T$
satisfying the properties in the statement of 
\autoref{main.thm.intro_CY}
is obtained by pulling back the universal Kuranishi family over 
$\Def(\mathcal X_0)$
through the germ of an analytic morphism 
$T \to \Def(\mathcal X_0)$ 
mapping 
$0 \in T$ 
to the distinguished point 
$\tilde 0 \in \Def(\mathcal X_0)$.
Subsequently, we utilize Hodge-theoretic methods to study the deformation theory of a general divisor 
$D_0$ 
in the linear system of a sufficiently high multiple of a very ample divisor on 
$\mathcal Y_0$ 
via the 
$T^1$-lifting 
criterion of Kawamata--Ran \cites{ran,Kaw92}: 
this result is used to show that the forgetful map from the deformation functor of the pair 
$(\mathcal X_0, D_0)$ 
is smooth onto its image in $\Def(\cX_0)$.  
Using the global invariant cycle theorem together with the vanishing condition \(H^2(\mathcal X_0, \mathcal O_{\mathcal X_0})=0\), we show that the forgetful map is dominant. In particular, every deformation of \(\mathcal{X}_0\) lifts to a deformation of the pair \((\mathcal X_0, D_0)\).
In view of this, \autoref{main.thm.intro_CY} implies that if generalized abundance holds for a nef line bundle 
$L_0$ on $\mathcal X_0$, then it holds also for all deformations of $L_0$ over 
$\Def(\mathcal X_0)$.

\medskip

It is natural to ask  whether 
\autoref{main.thm.intro_CY} 
could be generalized to $K$-torsion varieties without assuming the cohomology vanishing condition. However, the non-simple abelian varieties and elliptic K3 surfaces examples already described show that fibrations need not deform, so a refined formulation is required. Motivated by the relation with semiample line bundles, we restrict to deformations $\mathcal X\to T$ of $X$ for which a semiample line bundle $L$ inducing the fibration deforms, i.e., there exists a line bundle $\mathcal L$ on $\mathcal X$ with $\mathcal L|_{\mathcal X_0}=L$. We then ask whether the fibration structure also deforms.
 
\medskip

In our second main result, we show that this refined version of the deformation problem for fibrations has an affirmative answer. 
More precisely, we prove that, in deformations of a pair $(X,L)$ such that $X$ is a $K$-torsion variety and $L$ a semiample line bundle, the semiampleness of $L$ persists up to numerical equivalence on all fibers of the deformation.

\begin{theorem}[cf.\ Theorem~2.1]
\label{main.thm.intro_general}
Let 
$\pi \colon \mathcal{X} \to T$ 
be a smooth proper family of $K$-torsion K\"ahler manifolds over an analytic germ $(0\in T)$. Let $\mathcal L$ be a line bundle on 
$\mathcal{X}$ for which $\mathcal L\vert_{\mathcal X_0}$
is semiample. Then there is a $\pi$-semiample line bundle $\cM$ on $\cX$ with $\cM|_{\cX_0}\cong \cL|_{\cX_0}$ (and therefore $\cM\equiv_T\cL$).
In particular, there exists a commutative diagram of proper morphisms
\[\begin{tikzcd}[ampersand replacement=\&]
	\cX \&\& \cY \\
	\& T
	\arrow["\psi", from=1-1, to=1-3]
	\arrow["\pi"', from=1-1, to=2-2]
	\arrow["\phi", from=1-3, to=2-2]
\end{tikzcd}\]
factoring $\pi$, 
and a $\phi$-ample line bundle 
$\mathcal N$ on $\mathcal Y$ such that 
$\psi^\ast \mathcal N|_{\cX_0}\cong \cL^{\otimes m}|_{\cX_0}$ for some $m>0$, where $\phi$ is a flat projective family. If, moreover, $\pi$ and $\cL$ are algebraizable, then the above diagram is algebraizable, and the line bundles $\cM$ and $\cN$ can be chosen to be algebraic.
\end{theorem}

The $\pi$-semiample deformation $\cM$ of $\cL|_{\cX_0}$ is in fact homologically equivalent to $\cL$.  Note that $\cL$ itself may well not be $\pi$-semiample. Indeed, even for abelian varieties, the relative semiampleness of an arbitrary deformation $\cL$ of a semiample line bundle $\cL|_{\cX_0}$ may fail. For example, the deformation of the trivial line bundle induced by the Poincar\'e bundle on dual abelian varieties provides such a counterexample; see \autoref{rmk:numericalequiv-is-necessary}.

\medskip

The proof of \autoref{main.thm.intro_general} proceeds by passing to a suitable \'etale cover and using the Beauville--Bogomolov decomposition. Thus, we are reduced to proving that \autoref{main.thm.intro_general} holds for the different factors of the decomposition:
strict Calabi--Yau varieties, irreducible symplectic, and abelian factors separately.
While the strict Calabi--Yau case follows at once from \autoref{main.thm.intro_CY}, the irreducible holomorphic symplectic case was proven by Matsushita \cite{Matsushita2016} who showed the persistence of Lagrangian fibrations under the type of deformations that we consider in ~\autoref{main.thm.intro_general}.
We also show that the theorem holds for smooth families of complex tori---a result certainly known to experts, for which we could not find a suitable reference.
The final step in the proof of \autoref{main.thm.intro_general} 
consists in showing that the numerically semiample line bundles constructed on the deformations of the various factors descend along the \'etale cover taken to invoke the Beauville--Bogomolov decomposition.

Analogously to the case of \autoref{main.thm.intro_CY}, \autoref{main.thm.intro_general} implies that if generalized abundance holds for a nef line bundle 
$L_0$ on $\mathcal X_0$, then it holds, up to numerical equivalence, for all deformations of $L_0$ 
in 
$\Def(\mathcal X_0)$.
It is important to remark, though, that in this case the locus in 
$\Def(\mathcal X_0)$
where 
$L_0$
deforms is in general a strict analytic subvariety.

\medskip

Finally, we study deformations of subvarieties with trivial normal bundle inside $K$-trivial varieties.
While such subvarieties often give rise to fibrations, we exhibit examples showing that their deformations can be obstructed in general.

\medskip

{\it Structure of the paper}.
The paper is organized as follows. 
In \autoref{section:2}, we prove the two main results on deformations result for fibrations on \( K \)-torsion varieties. 
In \autoref{section:3}, we discuss deformations of subvarieties with trivial normal bundle and provide several examples displaying the various obstructions in general.

%% file: deformation_arxiv_v1.tex
\section{Deforming fibrations on $K$-torsion varieties}\label{section:2}

A {\it fibration} is a proper morphism of normal analytic varieties $f \colon X\to Y$ such that $\cO_Y\cong f_\ast\cO_X$.
A {\it $K$-trivial variety} is a smooth proper variety such that $K_X \sim 0$. A {\it $K$-torsion variety} is a smooth proper variety such that $K_X \sim_{\mathbb Q} 0$.

An \emph{analytic germ} $(0 \in T)$ is the equivalence class of a complex analytic space $T$ together with a point $0 \in T$, where two pairs $(T,0)$ and $(T',0')$ are equivalent if there exist open neighborhoods $U \subset T$ of $0$ and $U' \subset T'$ of $0'$ that are isomorphic as analytic spaces. Accordingly, when working with a germ, we fix a representative and shrink it whenever necessary.

In the following, by a smooth proper family $\pi:\mathcal{X}\to (0\in T)$ of $K$-torsion K\"ahler (resp.\ projective) varieties over an analytic germ, we mean a smooth proper morphism whose fibers are $K$-torsion and K\"ahler (resp.\ projective). Note that for a smooth proper family over a germ, the special fiber is K\"ahler (resp. $K$-torsion) if and only if every fiber is.
\smallskip
The goal of this section is to prove the following theorem.

\begin{theorem}\label{thm:main_thm}
    Let $\pi \colon \mathcal{X} \to (0\in T)$ be a smooth proper family of $K$-torsion K\"ahler varieties over an analytic germ.
    Further assume that there is a line bundle $\mathcal{L}$ on $\mathcal{X}$ such that $\mathcal{L}|_{\mathcal{X}_0}$ is semiample.
    Then, there is a $\pi$-semiample line bundle $\mathcal{M}$ such that $\mathcal{M}|_{\cX_0}\cong \mathcal L|_{\cX_0}$. If, moreover, $\pi$ and $\cL$ are algebraizable, then $\cM$ may be taken algebraizable as well.
\end{theorem}

\begin{remark}\label{rmk:numericalequiv-is-necessary}
    As $R^2\pi_*\bZ_{\cX}$ is a local system on $T$, the two deformations $\cL$ and $\cM$ are homologically equivalent. However, in general, $\mathcal{L}$ itself might not be $\pi$-semiample. For example, let $E$ be an elliptic curve and let $\mathcal{L}$ be the Poincar\'e line bundle on $E \times E$ normalized so that $\mathcal{L}|_{E \times \{0\}} = \mathcal{O}_E$.
    Then, $\mathcal{L}$ is only homologically trivial along the fibers of the projection onto the second factor $\pi_2 \colon E \times E \to E$, but not relatively semiample.
    In this case, we may then take $\mathcal{M} = \mathcal{O}_{E \times E}$.
\end{remark}

From \autoref{thm:main_thm} and its proof, we draw the following corollaries.  
Here and throughout, we denote by $\mathcal{X}\to\Def(X)$ the miniversal deformation in the analytic category (that is, the Kuranishi family of $X$) which for smooth proper K\"ahler $X$ is a smooth proper family of K\"ahler varieties over an analytic germ.  
Note that if in addition $H^2(X,\mathcal{O}_X)=0$, then 
$\mathcal{X}\to\Def(X)$ is a projective family, by Kodaira's embedding theorem.

\begin{corollary}\label{cor:H2=0}
    Let $X$ be a smooth projective $K$-torsion $n$-fold with $H^2(X,\cO_X)=0$ and $L$ a semiample line bundle on $X$.
    Then $L$ deforms to a relatively semiample line bundle $\cL$ on the miniversal deformation $\cX \to \Def(X)$.
\end{corollary}

\begin{corollary}\label{cor:def_fibration}
Let $X$ be a smooth projective $K$-torsion $n$-fold with $H^2(X,\cO_X)=0$ admitting a fibration $f \colon X \to Y$ with $Y$ projective,
and let $\cX \to \Def(X)$ be the miniversal deformation of $X$. 
Then $f$ deforms to a fibration $\phi \colon \cX \to \cY$, where $\cY$ is a flat projective deformation of $Y$ over $\Def(X)$.
\end{corollary}

\subsection{Deforming fibrations on Calabi--Yau varieties}

The main technical step towards proving 
\autoref{thm:main_thm} in general
is the following special case of 
\autoref{thm:main_thm}.

\begin{proposition}\label{main_prop}
Let $\pi\colon \mathcal{X} \to (0\in T)$ and $\mathcal{L}$ be as \autoref{thm:main_thm} and let $X = \mathcal{X}_0$ and $L = \mathcal{L}_0$.  Then, the conclusion of \textup{\autoref{thm:main_thm}} holds under the additional assumptions that $$H^1(X, \cO_X) \ =\  H^2(X,\cO_X) \ =\  0 \quad \textup{and} \quad K_X \sim 0.$$
\end{proposition}

\medskip
\noindent The assumption $K_X \sim 0$
implies that 
$\Omega_X^{n-1} \cong T_X$ and thus
\begin{align*}
h^0(X,T_X)=h^0(X,\Omega_X^{n-1})=h^{n-1}(X,\mathcal{O}_X)=h^{1}(X,\mathcal{O}_X)=0,    
\end{align*}
where the second-to-last equality follows from Serre duality.
In particular, $X$ admits a miniversal deformation space $\Def(X)$ (see, e.g., \cites{Dou74,Grau74}),
and as $H^2(X,\cO_X)=0$, the line bundle $L \coloneqq \mathcal{L}_0$ extends to $\Def(X)$ (see, e.g., \cite{HartshorneDefTheory}*{Theorem 6.4}). Thus, without loss of generality, we may prove \autoref{main_prop} under the assumption that $\cX \to T$ is the miniversal deformation $\cX \to \mathrm{Def}(X)$.

By assumption, $L$ is semiample and defines a fibration $X \to Y$.
Up to replacing $L$ with a tensor power, $L$ is the pullback of a very ample line bundle on $Y$.
In what follows, we will write $D$ for a general element of $|L|$, which is smooth by Bertini's theorem.  Notice that, if $\dim Y =1$, we have $Y = \mathbb P ^1$ by the assumption $H^1(X,\cO_X)=0$.
Thus, by choosing $D$ to be the pullback of a general section of $\mathcal{O}_{\mathbb P ^1}(1)$, we can guarantee that $D$ is connected regardless of $\dim Y$.

In order to prove \autoref{main_prop}, we consider the forgetful map $\Def(X, D) \longrightarrow \Def(X)$ and show that:
\begin{enumerate}
\item the deformation space $\Def(X, D)$ is smooth and the forgetful map is smooth onto its image; 
and
\item the forgetful map is dominant.
\end{enumerate}

To accomplish the first point, we use the following lemma.

\begin{lemma}\label{lemma:1}
   Under the assumptions of \autoref{main_prop}, the deformation space $\Def(X,D)$ of the pair $(X,D)$ is smooth, and the forgetful map $$\Def(X,D)\ \longrightarrow \  \Def(X)$$ is smooth onto its image.
\end{lemma}

\begin{proof}
 This is originally proved in \cite{ran}*{Theorem 2.1}; we include a proof for the reader's convenience. For a deformation $(X',D')$ of $(X,D)$ over an Artinian ring $A'$, the deformation module is the hypercohomology group $\bH^1(X',T_{X'/A'}\to \cO_{D'}(D'))$, which,
  by the condition $K_X \sim 0$,
  is identified with 
  \begin{equation}\label{eq:identification}
  \bH^1(X',\Omega_{X'/A'}^{n-1}\to \omega_{D'/A'})\ \cong\  H^1(X',\Omega_{X'/A'}^{n-1}(\log D')(-D')).
  \end{equation}
  This can be identified with $\gr^{n-1}_FH^n_c(X'\setminus D',A')$, where $F^\bullet$ is the filtration coming from the Hodge--de Rham spectral sequence 
  \begin{equation}
  \begin{split}
      H^{q}(X',\Omega^p_{X'/A'}(\log D')(-D')) & \ \Rightarrow\  H^{p+q}(X',\Omega^\bullet_{X'/A'}(\log D')(-D'))\\
      & \ \simeq \  H^{p+q}_c(X'\setminus D',A'),
  \end{split}
  \end{equation}
  where the latter isomorphism follows from \cite{EV86}*{Proposition (A.2)}.
  Indeed, for $j: X' \setminus D' \to X'$ the inclusion, one has
  $$
  \Omega^\bullet_{X'/A'}(\log D')(-D')\ \simeq \ \mathbb{D}(\Omega^{\bullet}_{X'/A'}(\log D'))\ \simeq \ \mathbb{D}(Rj_{*}\bC_{X' \setminus D'}) \ \simeq j_{!}\bC_{X'\setminus D'}.
  $$
  By Deligne's argument in \cite{delignelefschetz}*{\S~3}, since the Hodge--de Rham spectral sequence degenerates on the special fiber, the complex $\Omega^\bullet_{X'/A'}(\log D')(-D')$ is flat, i.e., every cohomology sheaf of the complex is flat. Thus we have
  $$
  H^{p+q}_c(X',A')\ \cong\  H^{p+q}_c(X',\bC)\otimes A',
  $$
  the Hodge--de Rham spectral sequence degenerates over $A'$ as well, and the deformation module $\gr^{n-1}_FH^n_c(X'\setminus D',A')$ of $(X,D)$ is free and compatible with base change. By $T^1$-lifting, $\Def(X,D)$ is smooth.

Now we show the smoothness of $\Def(X,D)\to \Def(X)$. Notice that the deformation module of $X'$ over $A'$ is $\gr^{n-1}_F H^n(X',A')$. From the natural exact sequence

\begin{equation}
\begin{tikzpicture}[descr/.style={fill=white,inner sep=1.5pt}]
  \matrix (m) [
    matrix of math nodes,
    row sep=1em,
    column sep=2.5em,
    text height=1.5ex, text depth=0.25ex
  ]
  { H^{n-1}(X',A')& H^{n-1}(D',A') & H^n_c(X'\setminus D',A') \\
    H^n(X',A') & H^n(D',A') 
      & {} \\
  };
\path[overlay,->, font=\scriptsize,>=latex]
    
    (m-1-1) edge (m-1-2)
    (m-1-2) edge (m-1-3)
    (m-1-3) edge[out=355,in=175]  (m-2-1)
    (m-2-1) edge (m-2-2);
\end{tikzpicture}
\label{LEStop}
\end{equation} where the first and last maps are the restriction maps, we obtain an exact sequence
\begin{equation}\label{LESgr}
\begin{tikzpicture}[descr/.style={fill=white,inner sep=1.5pt}]
  \matrix (m) [
    matrix of math nodes,
    row sep=1em,
    column sep=2.5em,
    text height=1.5ex, text depth=0.25ex
  ]
  { \gr^{n-1}_FH^{n-1}(X',A') & \gr^{n-1}_FH^{n-1}(D',A')& \gr^{n-1}_FH^n_c(X'\setminus D',A') \\
       \gr^{n-1}_F H^n(X',A')& \gr^{n-1}_FH^n(D',A') .
      & {} \\
  };
\path[overlay,->, font=\scriptsize,>=latex]
    
    (m-1-1) edge (m-1-2)
    (m-1-2) edge (m-1-3)
    (m-1-3) edge[out=355,in=175] (m-2-1)
    (m-2-1) edge (m-2-2);
\end{tikzpicture}
\end{equation}
since the degeneration of the Hodge--de Rham spectral sequence implies the morphisms of \eqref{LEStop} are strictly compatible with the filtrations.  All of the modules in \eqref{LESgr} are free and compatible with base change, and it follows that the same is true for the kernel and image of each morphism.  Thus, the tangent map of $\Def(X,D)\to \Def(X)$ has constant rank, which implies smoothness onto the image.
\end{proof}
\begin{remark}
 Poincar\'e duality identifies $$H^n_c(X'\setminus D',A') \ \simeq\  H^n(X'\setminus D',A')^\vee,$$ which is compatible with the Hodge and weight filtrations. Thus, we could have just as easily used the above argument with the usual log-de Rham complex.
\end{remark}

Now we have smoothness for \autoref{main_prop}, and we want to show that the forgetful map is dominant.  Given \autoref{lemma:1}, to prove dominance, it suffices to show that the map on tangent spaces is surjective.
Given \eqref{eq:identification}, the map on tangent spaces is surjective if and only if the restriction map $$H^{n-1,1}(X) \ \longrightarrow \ H^{n-1,1}(D)$$ is the zero map. This is the content of the following Lemma.

\begin{lemma}\label{lemma:2}
Under the assumptions of \autoref{main_prop}, the restriction
$$\gr_F^{n-1}H^n(X,\bC)\ \longrightarrow \ \gr_F^{n-1}H^n(D,\bC)$$
is the zero map.    
\end{lemma}

\medskip
\begin{proof}
    We would like to argue by the global invariant cycles theorem.
    Since $D$ is not a fiber of $f$ whenever $\dim Y > 1$, we first perform a reduction step.
    
    Let $S$ be the complete linear series $|D|$ and $g \colon \cD\to S$ the universal family of divisors.
    Since $|D|$ is base-point-free, the natural evaluation map $\ev:\cD\to X$ is then a projective bundle
    over $X$.
    Let $S^\circ\subset S$ be the open subset parametrizing smooth divisors, and $g^\circ \colon \cD^\circ\to S^\circ$ the restriction family.
    By construction, $D$ appears as a fiber of $g^\circ$.
    
    By the evaluation map, $D \subset \mathcal{D}$ is sent isomorphically to $D \subset X$. Thus, by the global invariant cycle theorem, the morphism $\gr_F^{n-1}H^n(X,\bC)\to \gr_F^{n-1}H^n(D,\bC)$ factors as \[\begin{tikzcd}[ampersand replacement=\&]
	{\gr_F^{n-1}H^n(X,\bC)} \&\& {\gr_F^{n-1}H^n(D,\bC)} \\
	{\gr_F^{n-1}H^n(\mathcal{D},\bC)} \&\& {\gr_F^{n-1}H^0(S^\circ,R^ng^\circ_\ast\bC_{\cD^\circ})}.
	\arrow[from=1-1, to=1-3]
	\arrow[from=1-1, to=2-1]
	\arrow[from=2-1, to=2-3]
	\arrow[from=2-3, to=1-3]
\end{tikzcd}\]
    Thus, in order to conclude the result, it suffices to show $\gr_F^{n-1}H^0(S^\circ,R^ng^\circ_\ast\bC_{\cD^\circ}) = 0$. By the relative hard Lefschetz theorem, we have
    $$
    R^ng^\circ_\ast\bC_{\cD^\circ}\ \simeq\  R^{n-2}g^\circ_\ast\bC_{\cD^\circ}(-1).
    $$
    Then, it follows that $$\gr_F^{n-1}H^0(S^\circ,R^ng^\circ_\ast\bC_{\cD^\circ})\ \simeq \ \gr_F^{n-2}H^0(S^\circ,R^{n-2}g^\circ_\ast\bC_{\cD^\circ}).$$ We need to show that the latter is $0$. By the global invariant cycles theorem, one has the surjection $$\gr_F^{n-2}H^{n-2}(\mathcal{D},\bC)\ \twoheadrightarrow\ \gr_F^{n-2}H^0(S^\circ,R^{n-2}g^\circ_\ast\bC_{\cD^\circ}).$$
    To conclude, it is enough to show $H^0(\mathcal{D},\Omega_{\mathcal{D}}^{n-2})=0$. By Grauert's theorem and the fact that $\ev:\cD\rightarrow X$ is a projective bundle, we have $$R^i\pi_\ast\mathcal{O}_{\mathcal{D}}\ =\ 0\ \ \ \ \textup{for any } \ i > 0.$$
    Therefore, the Leray spectral sequence degenerates and $H^{n-2}(\mathcal{D},\mathcal{O}_{\mathcal{D}})\simeq H^{n-2}(X,\mathcal{O}_{X})$, and by duality and the assumption $K_X\sim \mathcal{O}_X$ we obtain
    $$h^{n-2,0}(\cD)\ = \
    h^{0,n-2}(\mathcal{D}) \ =\  
    h^{0,n-2}(X) \ =\  
    h^{n-2,0}(X)\ =\  
    h^{2,0}(X)\ =\ 0,
    $$ where we use the assumption $H^2(X,\cO_X)=0$.
\end{proof}

\smallskip

With \autoref{lemma:1} and \autoref{lemma:2}, we complete the proof of \autoref{main_prop}.

\begin{proof}[Proof of \autoref{main_prop}]
Recall from above that $f: X \to Y$ is the morphism induced by $L$, where $L$ is the pullback of a very ample line bundle on $Y$ and $D \in |L|$ is a general element, which is smooth.  Furthermore, $L$ deforms to a line bundle $\cL$ on the miniversal deformation $\cX \to \mathrm{Def}(X)$.

By \autoref{lemma:1} and \autoref{lemma:2}, we know that $\Def(X, D) \rightarrow \Def(X)$ is a smooth and dominant map. This implies that a general section $D$ of $L$ deforms to a divisor $D_t$ on $\cX_t$, for $t \in \Def(X)$ sufficiently close to $0 \in \Def(X)$.
Thus, $D_t$ is algebraically equivalent to a rational section of $\cL_t$.
By the assumption $H^1(X,\cO_X)=0$, we know that $\mathrm{Pic}(\cX_t)$ is a discrete group, and hence $D_t$ is a section of $\cL_t$.
Thus, a general section of $L$ lifts to $\cL$ on $\cX$, hence so does every section.  As the relative base locus of $|\cL|$ on the proper family $\cX \to \Def(X)$ does not include $\cX_0 = X$, it does not dominate $\Def(X)$ and hence $\cL$ is relatively semiample in a neighborhood of $0 \in \Def (X)$ which concludes the proof of \autoref{main_prop}.
\end{proof}

\subsection{Proof of \autoref{thm:main_thm}}
Before the proof, we show that it suffices to prove the required statement in \autoref{thm:main_thm} up to a power and up to numerical equivalence.
\begin{lemma}\label{lem:roots}
    Let $\pi:\cX\to (0\in T)$ be a smooth proper family over an analytic germ.  Let $\cL$ and $\cM$ be two line bundles on $\cX$ such that $\cM|_{\cX_0}\cong \cL^{\otimes m}|_{\cX_0}$ and $\cM\equiv_T \cL^{\otimes m}$ for some $m>0$. Then there exists a line bundle $\cL'$ on $\cX$ with $\cL'|_{\cX_0}\cong \cL|_{\cX_0}$ and $\cM^{\otimes n}\cong (\cL')^{\otimes mn}$ for some $n>0$. If, moreover, $\pi,\cL,$ and $\cM$ are algebraizable, then so is $\cL'$.
\end{lemma}
\begin{proof}
By replacing $\cM$ with $\cM^{\otimes n}$ (and $m$ with $mn$) for some $n>0$, we may assume that $\cM$ is homologically equivalent to $\cL^{\otimes m}$, i.e., their first Chern classes coincide in 
\[
H^2(\cX,\mathbb{Z}) \ \cong \ H^2(\cX_0,\mathbb{Z}).
\]
In particular, the line bundle $(\cL^\vee)^{\otimes m}\otimes\cM$ is homologically trivial, and thus defines a section of the analytic group over $T$
\[
\Pic^0(\cX/T)\ \coloneqq \ R^1\pi_*\cO_\cX / R^1\pi_*\mathbb{Z}_{\cX},
\]
which is a family of complex tori over $T$.  Moreover, this section specializes to 0. The multiplication-by-$m$ morphism 
\[
[m]: \Pic^0(\cX/T)\ \to\  \Pic^0(\cX/T)
\]
is a finite étale cover. Hence, by shrinking $T$, this cover admits a local section through $[\mathcal{O}_X]\in \Pic^0(\cX_0)$ and lifting the given section $\cL^{-m}\otimes\cM$. In other words, there exists a section $\mathcal{F}$ of $\Pic^0(\cX/T)$ such that
\[\mathcal{F}|_{\cX_0}\cong \mathcal{O}_{\cX_0}\qquad \textup{and}\qquad
\mathcal{F}^{\otimes m} \ \cong\  (\cL^\vee)^{\otimes m}\otimes\cM.
\]
Setting $\cL'=\cL\otimes \mathcal{F}$, we obtain $\cL'|_{\cX_0}\cong \cL|_{\cX_0}$ and $\cM \cong (\cL')^{\otimes m}$. The final claim is clear by the same argument in the algebraic category.
\end{proof}

We will also need to separately treat the case of families of complex tori.
\begin{lemma}\label{lemma:bundle_ab_var}
Let $\pi \colon \mathcal{X} \to (0\in T)$ be a family of complex tori over an analytic germ, and let $\mathcal{L}$ be a line bundle on $\mathcal{X}$ whose central fiber $\mathcal{L}_0  \coloneqq \mathcal{L}|_{\mathcal{X}_0}$ is semiample.  
Then there exists a $\pi$-semiample line bundle $\mathcal{M}$ on $\mathcal{X}$ such that $\cM_0\cong \cL_0$. If, moreover, $\pi$ and $\cL$ are algebraizable, then $\cM$ may be taken algebraizable.
\end{lemma}

\begin{proof}

For any $t \in T$, the Chern class $c_1(\mathcal{L}_t) \in H^2(\mathcal{X}_t,\mathbb{Z})$ defines a bilinear form
\[
H_1(\mathcal{X}_t,\mathbb{Z}) \times H_1(\mathcal{X}_t,\mathbb{Z}) \longrightarrow \mathbb{Z}
\]
via $\langle\alpha,\beta\rangle=c_1(\mathcal{L}_t)(\alpha,\beta)$ using $H^2(\mathcal{X}_t,\mathbb{Z})\cong \bigwedge^2H^1(\mathcal{X}_t,\mathbb{Z})$.  
Let
\[
\mathbb{V}_t \subseteq  H_1(\mathcal{X}_t,\mathbb{Z})
\]
be the subspace where the form is totally degenerate.
Since $c_1(\mathcal{L})$ varies locally constantly in families, each $\mathbb{V}_t$ is a
$\mathbb{Z}$-Hodge substructure and the rank of $\mathbb{V}_t$ is locally constant. The $\mathbb{V}_t$ assemble into a sub-variation of Hodge structure
\[
\mathbb{V} \subseteq (R^1 \pi_\ast \mathbb{Z})^\vee,
\]
which corresponds, via the correspondence between sub-Hodge structures of a Hodge structure of weight $-1$ and complex subtori, to a family of abelian subvarieties
\[
\mathcal{Y} \longrightarrow T
\]
inside the family $\mathcal{X} \to T$. Let
\[
q : \mathcal{X} \longrightarrow \mathcal{Z}
\]
be the quotient family $\mathcal{Z}  \coloneqq  \mathcal{X}/\mathcal{Y}$, which is again a family of complex tori over $T$. By construction, $c_1(\mathcal{L}_t)$ is trivial along the fibers of $\mathcal{Y}_t$ and therefore descends to a Hodge class on $H^2(\mathcal{Z}_t,\mathbb{Z})$, which is moreover a polarization because it is so on the special fiber.

Since $\mathcal{L}_0$ is semiample, some power $\mathcal{L}_0^{\otimes n}$ descends to an ample line bundle $\overline{\mathcal{F}}_0$ on $\mathcal{Z}_0$.  The numerical class of $\overline{\mathcal{F}}_0$ stays Hodge in the family $\cZ/T$, so the analytic space of line bundles in this numerical class is a $\Pic^0(\mathcal{Z}/T)=\mathcal{Z}^{\vee}/T$-torsor, and in particular smooth over $T$.  Thus, there is a line bundle $\cN$ on $\cZ$ in this numerical class with $\cN_0\cong \overline{\mathcal{F}}_0$.  Moreover, by the above $\cN$ is relatively ample.

Define $\mathcal{M}  \coloneqq  q^{*}\mathcal{N}$. Then by construction:
\begin{itemize}
    \item $\mathcal{M} \equiv_T \mathcal{L}^{\otimes n}$, since both yield the same induced polarization on the quotient;
    \item on the central fiber,
  \[
  \mathcal{M}_0 = q_0^{*}\mathcal{N}_0
  \cong q_0^{*}\overline{\mathcal{F}}_0
  = \mathcal{L}_0^{\otimes n};
  \]
  \item $\mathcal{M}$ is relatively semiample because it is the pullback of the relatively ample line bundle $\mathcal{N}$ along an abelian fibration.
\end{itemize}
Finally, applying Lemma~\ref{lem:roots}, the lemma is proven.
\end{proof}

Now let us complete the proof of \autoref{thm:main_thm}.

\begin{proof}[Proof of \autoref{thm:main_thm}]
    For brevity, set \(X \coloneqq \mathcal{X}_0\).
    Let \(X' \to X\) be an \'{e}tale cover that splits as a product
    \[
        X' \;=\; \prod X_i'
    \]
    of strict Calabi–-Yau varieties,  irreducible holomorphic symplectic varieties, and complex tori (cf.~\cite{Bea83}*{Theorem 1}).
    After passing to a further étale cover, we may assume that \(X' \to X\) is a Galois cover with Galois group \(G\) (cf.~\cite{stacks-project}*{Tag~0BN2}).

    The cover $X' \to X$ is determined topologically.
    Then, since $\cX \to T$ is topologically trivial and $T$ is contractible, the natural map $\pi_1(X)\to\pi_1(\mathcal{X})$ is an isomorphism and thus the cover $X' \to X$ extends to an analytic cover $q:\cX' \to \cX$, inducing a deformation of $X'$ over $T$.
    By pulling back a K\"ahler form, each fiber of $\cX'\rightarrow T$ is again K\"ahler.
    
    We claim that \(\mathcal{X}'\) also splits as a product.
    It suffices to show that a versal deformation of \(X'\) preserves the product structure, in which case the assertion follows by base change; see \cite{BL22}*{\S4.2}.
    Each factor \(X'_i\) has unobstructed deformations of dimension \(h^1(X'_i,T_{X'_i})\), and \(X'\) has unobstructed deformations of dimension \(h^1(X',T_{X'})\).
    Since deformations of the factors produce deformations of the product, it remains to show
    \[
        h^1(X',T_{X'}) \;=\; \sum_i h^1(X'_i,T_{X'_i}).
    \]
    This follows from the Künneth formula, using that all non-abelian factors satisfy the conditions
    \(h^0(X'_i,T_{X'_i})=0\) and \(h^1(X'_i,\mathcal{O}_{X'_i})=0\).
    Hence
    \[
        \mathcal{X}' \;=\; \prod \mathcal{X}'_i .
    \]
 Let \(\mathcal{L}'\) be the pullback of \(\mathcal{L}\) to \(\mathcal{X}'\).
    Since the strict Calabi--Yau and symplectic factors have vanishing irregularity, \(\Pic(\mathcal{X}')\) splits accordingly,
    so
    \[
        \mathcal{L}' = \boxtimes\, \mathcal{L}'_i.
    \]
    Thus \(\mathcal{L}'\) is relatively semiample if and only if each \(\mathcal{L}'_i\) is relatively semiample.

    For the strict Calabi--Yau factors, relative semiampleness follows from \autoref{main_prop}. For an irreducible symplectic factor \(X'_i\), \(L'_i\) is either big or induces a Lagrangian fibration (cf.~\cite{Mat99}*{Theorem~2}).  If \(L'_i\) defines a Lagrangian fibration, relative semiampleness follows from \cite[]{Matsushita2016}*{Theorem~1.2}.  If instead $L_i'$ is big, consider the miniversal deformation $\pi_i':(\mathcal{X}_i',\mathcal{L}_i')\to \Def(X_i',L_i')$ (see \cite[]{Matsushita2016}*{Theorem~1.1}).  Since $X'_i$ is projective (see, e.g., \cite{huybrechts_basic}*{Theorem 3.11}), by the Kawamata--Viehweg vanishing and cohomology and base change, $\pi_{i*}'\mathcal{L}_i'$ is locally free and compatible with base change, hence $\mathcal{L}_i'$ is relatively semiample.

    It remains to treat the complex torus factor \(\mathcal{X}'_{i_0}\), if present.
    By \autoref{lemma:bundle_ab_var}, there exists
    \[
        \mathcal{M}'_{i_0} \equiv_T \mathcal{L}'_{i_0}
    \]
    with \(\mathcal{M}'_{i_0}\) relatively semiample and $\cM'_{i_0}|_X\cong \cL_{i_0}'|_X$.
    Set
    \[
        \mathcal{M}' \coloneqq
            \mathcal{M}'_{i_0} \boxtimes
            \Bigl(\boxtimes_{i \neq i_0} \mathcal{L}'_i\Bigr).
    \]
    Then \(\mathcal{M}'\) is relatively semiample, satisfies \(\mathcal{M}' \equiv_T \mathcal{L}'\), and agrees with \(\mathcal{L}'\) on the central fiber. It follows that $\cM \coloneqq (\det q_*\cM')^{\otimes 2}$ is relatively semiample, numerically equivalent to $(\det q_*\cL')^{\otimes 2}=\cL^{\otimes 2|G|}\otimes(\det q_*\cO_{\cX'})^{\otimes 2}=\cL^{\otimes 2|G|}$, since the square of the determinant of the left regular representation of $G$ is trivial,
    and $\cM$ agrees with $\cL^{\otimes 2|G|}$ on the special fiber. The theorem then follows from \autoref{lem:roots}.
\end{proof}

Now, we prove the corollaries stated in the beginning of this section.

\begin{proof}[Proof of \autoref{cor:H2=0}]
    Let $\pi: \cX \to \Def(X)$ be the miniversal deformation of $X$. As $H^2(X,\cO_X) = 0$, the line bundle $L$ on $X$ deforms to a line bundle $\cL$ on $\cX$ (see, e.g., \cite{HartshorneDefTheory}*{Theorem 6.4}). Then the statement follows immediately from \autoref{thm:main_thm}.
\end{proof}

\begin{proof}[Proof of \autoref{cor:def_fibration}]
    Let $L$ be a semiample line bundle on $X$ inducing the fibration $f \colon X\to Y$. By \autoref{cor:H2=0},  $L$ deforms to a relatively semiample line bundle $\cL$ on $\cX$ over $\Def(X)$. Up to taking a tensor power, one can assume that $\cL$ is base-point free over $\Def(X)$, and let $\phi \colon \cX \to \cY$ be the map induced by $\cL$. 
    
    To conclude, we need to show that $p \colon \cY \to \Def(X)$ is flat, and it suffices to show that $\pi_{*}\cL$ is flat since $\cY=\Proj_{\Def(X)}\left(\oplus_{m\geq0}\pi_*\cL^{\otimes m}\right)$. Here, as $H^2(X,\cO_X)=0$, the miniversal deformation space $\Def(X)$ is algebraizable. The proof of \autoref{main_prop} shows that the map \[(\pi_\ast\cL^{\otimes m})\otimes k(0)\ \to\  H^0(X,L^{\otimes m})\] is surjective for any $m>0$, and it follows from cohomology and base change that $\pi_\ast\cL^{\otimes m}$ is locally free for any $m>0$, and in particular it is flat.
\end{proof}

To conclude this section, we prove \autoref{main.thm.intro_general}.

\begin{proof}[{Proof of \autoref{main.thm.intro_general}}]
    We utilize the notation and setup of \autoref{thm:main_thm}.
    Then, by \autoref{thm:main_thm}, to conclude, we only need to show that $\cY \to T$ is flat.

    First, we analyze the fibration defined by $\sum_{g \in G}g^\ast\cM '$ on $\cX'$.
    We let $\cY'$ denote the image of said fibration.
    By the splitting $\cL' = \boxtimes \cL'_i$ and the definition of $\sum_{g \in G}g^\ast\cM '$, $\cY'$ splits as a product $\cY' = \prod \cY'_i$, where $\cX_i' \to \cY'_i$ is the fibration defined by $\cL'_i$.
    If the factor $\cX'_i$ is a strict Calabi--Yau factor, $\cY'_i \to T$ is flat by \autoref{cor:def_fibration}.
    If the factor $\cX'_i$ is a complex tori factor, then $\cY'_i \to T$ is a family of complex tori and is therefore flat.
    Lastly, if $\cX'_i$ is an hyperk\"ahler factor, then $\cY'_i \to T$ is flat by \cite{Matsushita2016}*{Theorem 1.2} in the case $\cL_i'$ induces a Lagrangian fibration, and by
    \cite{wahl_equi}*{Theorem 1.4}, which works in the analytic category as well, in the case when it induces a birational contraction, using that the contraction has symplectic and therefore rational singularities.
    Thus, by taking the product, it follows that $\cY' \to T$ is flat.

    Then, by construction, $\cY$ is a $G$-quotient of $\cY'$---let $q:\cY'\to\cY$ be the quotient morphism.  From the existence of a trace map, $\cO_{\cY}$ is a split summand of $q_\ast\cO_{\cY'}$, so $\cY$ is flat.
\end{proof}

%% file: subvarieties_arxiv_v1.tex
\section{Deformation of subvarieties with trivial normal bundle} \label{section:3}

Let $X$ be a smooth strict Calabi--Yau variety and $Y \subset X$ a smooth divisor that is itself a $K$-trivial variety.
Then, by \cite{ran}*{Theorem 2.1}, the embedding $\iota \colon Y \hookrightarrow X$ is unobstructed.
In particular, since $h^0(Y,N_{Y/X})=h^0(Y,\omega_Y)=1$, it follows that $Y$ deforms in $X$ in a 1-dimensional family of subvarieties of $X$.
Since $X$ is strict Calabi--Yau, these subvarieties need to be linearly equivalent.
In particular, $Y$ moves in a basepoint free pencil, thus defining a morphism $X \to \mathbb P ^1$. Inspired by this observation, we investigate the following question:

\begin{question} \label{question}
    Let $X$ be a smooth $K$-trivial $n$-fold with $n\geq 3$, and $Y\subseteq X$ be a (smooth) subvariety with trivial normal bundle in $X$. 
    \begin{enumerate}
    \item \label{question:1} Are the deformations of $Y$ unobstructed? 
    \item \label{question:2} If yes, is $Y$ a fiber of some fibration of $X$?
    \end{enumerate}
\end{question}

The following example shows that \autoref{question}\autoref{question:2} has a negative answer in general.

\begin{example}[A hyperk\"{a}hler of K$3^{[2]}$-type]
Let $S$ be a K3 surface admitting an elliptic fibration $\phi \colon S\rightarrow \bP^1$.
Then the Hilbert scheme of two points $ S^{[2]}$ is a hyperk\"{a}hler 4-fold which admits a Lagrangian fibration
$$
\psi \ \colon \ S^{[2]} \ \longrightarrow \ (\bP^1)^{[2]}\simeq \bP^2
$$
whose general fiber is an abelian surface of the form 
$E_1\times E_2$, 
where 
$E_i$ 
is isomorphic to a fiber of 
$\phi$;
see \cite{Bea83} for details.
We denote by 
$\Delta \subset S^{[2]}$
the locus of length $2$ $0$-dimensional subschemes of $S$ supported at just one point.

Let $E$ be a smooth fiber of $\phi$, and $p\in S$ be a point away from $E$. Then $E+p$ gives an embedding of $E$ into $S^{[2]}\setminus \Delta$,
with image still denoted by $E+p$. Note that the preimage of $E$ under the \'{e}tale double cover $$\tau\ : \ (S\times S)\setminus \Delta_S \ \longrightarrow \ S^{[2]}\setminus\Delta $$ is $E\times\{p\}\sqcup \{p\}\times E$, and hence $\tau$ induces an isomorphism on formal neighborhoods of $E\times\{p\}$ and $E+p$. Then the triviality of the normal bundle of $E+p$ follows from that of $E\times\{p\}$ in $S\times S$.

By Matsushita's theorem \cite{Mat99}, $E$ cannot be a fiber of any fibration of $S^{[2]}$. However, up to an alteration $X$ of $S^{[2]}$, $E$ is a fiber of some fibration of $X$:
indeed, the double cover $S\times S$ of the symmetric product $S^{(2)}$ admits a fibration to $S\times \bP^1$ with one fiber isomorphic to $E$, and hence so does the corresponding double cover of $S^{[2]}$.
    
\end{example}

The following example shows how, in general, \autoref{question}\autoref{question:1} also has a negative answer: there are $K$-trivial varieties admitting $K$-trivial subvarieties with trivial normal bundle and obstructed deformations.

\begin{example}[A K3-fibered strict Calabi--Yau 3-fold]
    Let $D \subset \bP^3_{[x:y:z:w]} \times \bP^1_{[t_0:t_1]}$ be a divisor in the linear system $|\mathcal{O}(4,2)|$ with equation 
    \[
    D = (t_0^2 f_1(x,y,z,w) + t_0t_1 f_2(x,y,z,w) + t_1^2 f_3(x,y,z,w) = 0)
    \]
    where $f_i \in H^0(\bP^3, \mathcal{O}(4))$.
    Further assume that $f_1$ is the equation of a smooth quartic surface containing a line and $f_3$ is very general.
    We claim that $D$ may be chosen to be smooth.
    Granted this fact,
    by adjunction and the Lefschetz hyperplane theorem,
    $D$ is a strict Calabi--Yau 3-fold
    and via the projection $pr_2 \colon \bP^3_{[x:y:z:w]} \times \bP^1_{[t_0:t_1]} \to \bP^1$, $D \to \bP^1$ is a K3-fibration.
    Let $D_{[t_0:t_1]}$ denote the fiber over $[t_0:t_1]$.

    By construction, $D_{[0:1]}$ is a very general quartic K3 surface and so has Picard rank 1, and hence the same holds for the generic fiber of $D \to \mathbb P ^1$.
    The fiber $D_{[1:0]}$ is a quartic that contains a line and thus has an elliptic fibration: let $\langle H_1, H_2 \rangle \subset \bP^3$ be the pencil of hyperplanes containing the line.
    For any $H \in \langle H_1, H_2 \rangle$, the intersection $H \cap D_{[1:0]}$ is a reducible plane quartic that factors as the line together with a plane cubic $C$.
    By adjunction, $C^2 = 0$, so the family of residual cubics on $D_{[1:0]}$ gives an elliptic fibration on $D_{[1:0]}$.  Let $C$ be one of these elliptic curves.
    Choose $f_2$ in the equation of $D$ so that $C \subset (f_2 = 0)$.
    
    Now, we show that $D$ can be chosen to be smooth, for $f_1, f_2$ as above and $f_3$ very general.
    Consider the linear system $V \subset |\mathcal O (4,2)|$ spanned by all such $D$'s where $f_3$ is any quartic equation.
    By construction, the base locus of $V$ coincides with $D_{[1:0]}$.
    Thus, by Bertini's theorem applied to a resolution of the linear system $V$, $V$ is free away from $D_{[1:0]}$.
    In particular, for a general choice of $D$ (e.g., when $f_3$ determines a very general quartic surface), $D$ is smooth away from $D_{[1:0]}$.
    In turn, for such general choice, $D_{[1:0]}$ is a smooth Cartier divisor in $D$, and hence $D$ is also smooth along $D_{[1:0]}$. Therefore, we conclude $D$ is smooth and $D \to \mathbb P^1$ is a K3-fibered strict Calabi--Yau 3-fold as claimed.
 
    By construction, the elliptic curve $C \subset \bP^3 \times \bP^1$ is a complete intersection of three divisors:
    the fiber $\mathbf{P} = \bP^{3} \in |\mathcal{O}_{\bP^3 \times \bP^1}(0,1)|$ over $[1:0]$, a relative hyperplane section $\mathbf{H} = H \times \bP^1 \in |\mathcal{O}_{\bP^3 \times \bP^1}(1,0)|$, and a divisor $\mathbf{C} \in |\mathcal{O}_{\bP^3 \times \bP^1}(3,0)|$ where $\mathbf{C} \subset \bP^3 \times \bP^1$ is the pullback from $\bP^3$ of any cubic surface containing $C$ (for example, the cone over $C$).
    In particular, since $C$ is a complete intersection, its normal bundle splits as a direct sum of the restriction of the normal bundles of these three divisors to $C$.
    As $\mathcal{O}_{\bP^3 \times \bP^1}(a,b)|_{C} = \mathcal{O}_C(3a)$, we have
    \[
    \mathcal{N}_{C/\bP^3 \times \bP^1} = \mathcal{O}_C \oplus \mathcal{O}_C(3) \oplus \mathcal{O}_C(9) \quad \text{ and } \quad \mathcal{N}_{C/\mathbf{P}} = \mathcal{O}_C(3) \oplus \mathcal{O}_C(9).
    \]

    Consider the commutative diagram 
    \begin{center}
    \begin{tikzcd}
        & 0 \arrow[d] & 0 \arrow[d] &  & \\
        0 \arrow[r] & \mathcal{N}_{C/D_{[1:0]}} \arrow[r] \arrow[d] & \mathcal{N}_{C/\mathbf{P}} \arrow[r] \arrow[d]  & \mathcal{N}_{D_{[1:0]}/\mathbf{P}}|_C \arrow[r]  & 0 \\
        0 \arrow[r] & \mathcal{N}_{C/D} \arrow[r] \arrow[d] & \mathcal{N}_{C/\bP^3 \times \bP^1} \arrow[r] \arrow[d] & \mathcal{N}_{D/\bP^3 \times \bP^1}|_C \arrow[r]  & 0 \\
         & \mathcal{N}_{D_{[1:0]}/D}|_C  \arrow[d] & \mathcal{N}_{\mathbf{P}/\bP^3 \times \bP^1}|_C \arrow[d] &   & \\
        & 0  & 0  &  & \\
    \end{tikzcd}
\end{center}
From exactness of the first two columns and first two rows, we can extend the diagram by a vertical arrow $\mathcal{N}_{D_{[1:0]}/\mathbf{P}}|_C \to\mathcal{N}_{D/\bP^3 \times \bP^1}|_C  $  (or, observe that $D_{[1:0]}|_C = D|_C$ to get the right most vertical map) and a horizontal arrow $\mathcal{N}_{D_{[1:0]}/D}|_C  \to \mathcal{N}_{\mathbf{P}/\bP^3 \times \bP^1}|_C$ (or, observe that $\mathcal{O}_D(D_{[1:0]}) = \mathcal{O}_{\bP^3 \times \bP^1}(\mathbf{P})|_D$).
Note that $\mathcal{N}_{C/D_{[1:0]}} = \mathcal{O}_C$ and $\mathcal{N}_{D_{[1:0]}/D}|_C = \mathcal{O}_C$.
As $D_{[1:0]} \subset \mathbf{P} =\bP^3$ is a quartic surface, $\mathcal{N}_{D_{[1:0]}/\mathbf{P}}|_C = \mathcal{O}_{\bP^3}(4)|_C = \mathcal{O}_C(12)$.  Filling in the remaining terms from $\mathcal{O}_{\bP^3 \times \bP^1}(a,b)|_{C} = \mathcal{O}(3a)$, we claim we have a commutative diagram
\begin{center}
    \begin{tikzcd}
        & 0 \arrow[d] & 0 \arrow[d] & 0 \arrow[d] & \\
        0 \arrow[r] & \mathcal{O}_C \arrow[r] \arrow[d] & \mathcal{O}_C(3) \oplus \mathcal{O}_C(9) \arrow[r] \arrow[d] & \mathcal{O}_C(12) \arrow[r] \arrow[d] & 0 \\
        0 \arrow[r] & \mathcal{N}_{C/D} \arrow[r] \arrow[d] & \mathcal{O}_C \oplus \mathcal{O}_C(3) \oplus \mathcal{O}_C(9) \arrow[r] \arrow[d] & \mathcal{O}_C(12) \arrow[r] \arrow[d] & 0 \\
        0 \arrow[r] & \mathcal{O}_C \arrow[r] \arrow[d] & \mathcal{O}_C \arrow[r] \arrow[d] & 0 \arrow[r]  \arrow[d] & 0\\
        & 0  & 0  & 0 & \\
    \end{tikzcd}
\end{center}

Observe that the map $\mathcal{O}_C \to \mathcal{O}_C$ in the bottom row is an isomorphism.
Indeed, if not, the map $\mathcal{N}_{C/D} \to \mathcal{O}_C \oplus \mathcal{O}_C(3) \oplus \mathcal{O}_C(9)$ would lift to $\mathcal{O}_C(3) \oplus \mathcal{O}_C(9)$.
By the middle row, $\mathcal{N}_{C/D}$ is a saturated subsheaf of $\mathcal{O}_C \oplus \mathcal{O}_C(3) \oplus \mathcal{O}_C(9)$.
In particular, the lift would lead to an isomorphism between $\mathcal{N}_{C/D}$ and $\mathcal{O}_C(3) \oplus \mathcal{O}_C(9)$, which is impossible, since these two sheaves have different degrees.
Similarly, we obtain that the map $\mathcal{O}_C(12) \to \mathcal{O}_C(12)$ in the third column is an isomorphism as well, which gives commutativity of the diagram.

To prove $\mathcal{N}_{C/D}$ is trivial, it suffices to show the map $\mathcal{O}_C \oplus \mathcal{O}_C(3) \oplus \mathcal{O}_C(9) \to \mathcal{O}_C(12)$ in the middle row is trivial on the $\mathcal{O}_C$ factor.  Indeed, this will show that the splittings of the right two vertical sequences commute which gives a splitting of the first vertical sequence, i.e., $\mathcal{N}_{C/D} \cong \mathcal{O}_C \oplus \mathcal{O}_C$.

To show this, consider the map $\mathcal{O}_C \oplus \mathcal{O}_C(3) \oplus \mathcal{O}_C(9) \to \mathcal{O}_C(12)$ restricted to the map $\mathcal{O}_C \to \mathcal{O}_C(12)$.  By construction from the splitting of the middle vertical sequence, this is the map $\mathcal{N}_{\mathbf{P}/\bP^3 \times \bP^1}|_C \to \mathcal{N}_{D/\bP^3 \times \bP^1}|_C$, which takes a normal vector to $C$ that is normal to $\mathbf{P}$ and sends it to its image in the space of normal vectors to $D$.  By direct computation of the normal vectors and tangent space, by choice of $f_2$, the normal vectors to $C$ that are normal to $\mathbf{P}$ are tangent to $D$, and therefore this map is trivial, as desired.

We conclude that the normal bundle $\mathcal{N}_{C/D}$ is trivial.
Yet, the embedded deformations of $C$ in $D$ are obstructed.
Indeed, for $[t_1:t_2]$ very general, $D_{[t_1:t_2]}$ has Picard rank 1 and hence cannot contain an elliptic curve.
The equation of $D$ accounts for this behavior:
infinitesimally, if $t_1^2 = 0$ but $t_1 \ne 0$, both $f_1$ and $f_2$ vanish along $C$, and so $C$ ``deforms'' in $D$ to first order;
on the other hand, if $t_1^2 \ne 0$, as $f_3$ does not vanish on $C$, $C$ cannot deform inside $D$ past first order.
\end{example}